\documentclass[11pt]{article}

\usepackage[latin1]{inputenc} \usepackage[T1]{fontenc} \usepackage{lmodern}
\usepackage{amsmath,amssymb,amsfonts}
\usepackage{pgf,color}
\usepackage{pgfarrows}
\usepackage{pifont}
\usepackage[colorlinks=true,linkcolor=blue,citecolor=blue,bookmarks=false,pdftex]{hyperref}
\usepackage[english]{babel}

\usepackage{fancyhdr}
\usepackage{amsthm}
\usepackage{amsmath}		
\usepackage{multicol}
\usepackage[all]{xy}
\usepackage{graphicx}
\usepackage{ulem}
\usepackage{stmaryrd}
\usepackage{pdfsync}

\providecommand{\myproofname}{Proof}

\let\origItemize=\itemize
\def\NoSpacing{
  \itemsep=0pt
  \parskip=0pt
  \parsep=0pt
  \partopsep=0pt
  \topsep=0pt
}
\renewenvironment{itemize}{\origItemize\NoSpacing}{\endlist}

\newcommand{\da}{$\ddot{\textrm{a}}$}
\newcommand{\uu}{$\ddot{\textrm{u}}$}
\newcommand{\bo}{}

\newcommand{\mF}{\mathbb{F}}

\newcommand{\mQ}{\mathbb{Q}}

\newcommand{\mZ}{\mathbb{Z}}
\newcommand{\mA}{\mathbb{A}}
\newcommand{\mG}{\mathbb{G}}

\newcommand{\boP}{\mathcal{P}}

\newcommand{\nequiv}{\not\equiv}

\newcommand{\eq}{\Leftrightarrow}

\newcommand{\geqs}{\geqslant}
\newcommand{\leqs}{\leqslant}

\newcommand{\eps}{\varepsilon}
\newcommand{\phy}{\varphi}

\newcommand{\pgcd}{\mathrm{pgcd}}

\newcommand{\Res}{\mathrm{Res}}

\newcommand{\Gal}{\mathrm{Gal}}

\newcommand{\Kern}{\mathrm{Ker}}

\newcommand{\N}{\mathrm{N}}

\newcommand{\GL}{\mathrm{GL}}

\newcommand{\tx}{\textrm}

\newcommand{\dcases}{\begin{cases}}
\newcommand{\fcases}{\end{cases}}
\newcommand{\ditem}{\begin{itemize}}
\newcommand{\fitem}{\end{itemize}}
\newcommand{\dequa}{\begin{equation}}
\newcommand{\fequa}{\end{equation}}

\newcommand{\FF}[1]{\mF_{#1}}
\newcommand{\FFx}[1]{\FF{#1}^{\times}}

\newcommand{\TT}[1]{T_{#1}(\FF{q})}

\newtheorem{dfn}{Definition}

\newtheorem{thm}{Theorem}
\newtheorem{cor}{Corollary}
\newtheorem{lem}{Lemma}
\newtheorem{prop}{Proposition}

\begin{document}

\normalem

\author{Cl\'ement Dunand}

\date{\today}
\title
 {On Modular Inverses of Cyclotomic Polynomials \\and the Magnitude of their Coefficients}

\maketitle

\begin{abstract}
  Let $p$ and $r$ be two primes and $n$, $m$ be two distinct divisors of
  $pr$. Consider $\Phi_{n}$ and $\Phi_{m}$, the $n$-th and $m$-th
  cyclotomic polynomials. In this paper, we present lower and upper bounds for
  the coefficients of the inverse of $\Phi_{n}$ modulo $\Phi_{m}$ and mention an application to torus-based cryptography.
\end{abstract}


\section{Introduction}

The magnitude of coefficients of polynomials derived from cyclotomic
polynomials has attracted attention since the 19th century.  If $\phy$ denotes
the Euler totient function, the $n$-th cyclotomic polynomial $\Phi_n$ is a monic polynomial of degree $\phy(n)$ whose roots are the primitive $n$-th roots
of unity. In the following, we denote by $(a_i)_{0\leqslant i\leqslant \phy(n)}$ its coefficients.\smallskip

Many results published so far deal with $\Phi_n$. On one hand, we have
asymptotic results which outline that these coefficients may have an
exponential behaviour for infinitely many $n$ (see for instance
Erd\"os~\cite{erdos} or Bateman~\cite{bateman}). On the other hand, there
exist numerous studies for integers $n$ having only few prime factors.
In this way, Migotti \cite{migo} showed in 1883 that if $n$ is composed of at
most two primes $p$ and $r$, the coefficients of $\Phi_{pr}$ can only be $-1$, $0$ 
or $1$. Later, around 1965, Beiter \cite{beit} and Carlitz
\cite{car} exhibit more precise criteria for these coefficients to be $0$ or
$\pm 1$. More recently in 1996, Lam and Leung~\cite{lam} give these
coefficients in an explicit way.

The first example of a cyclotomic polynomial with a coefficient of magnitude 2
is $\Phi_{105}$, whose 7th and 41st coefficients are -2. Yet, when $n$ is the product of few primes, we can still find interesting bounds for the coefficients of $\Phi_n$. For $n$ product of three distinct primes $p<q<r$, Bang~\cite{bang}
showed in 1895 that $|a_i|\leqs p-1$. Later, in 1968, Beiter \cite{beit} and Bloom \cite{bloom} gave a better
bound when $q$ or $r$ equals $\pm 1$ modulo $p$, that is $|a_i|\leqs (p+1)/2$. The
conjecture that this bound could hold for all prime numbers $p$, $q$, $r$ has
recently been proved to be wrong by Gallot and Moree in \cite{moree}. Bachman
\cite{bach} gave a better bound in 2003: for any distinct primes $p<q<r$, $|a_i|\leqs
p-\lceil p/4\rceil $.  In 1968, Bloom \cite{bloom} even gave a bound for
 a product of four distinct primes: for $n=pqrs$ with $p<q<r<s$, we have $|a_i|\leqs p(p-1)(pq-1)$.

Moree has recently studied cofactors of cyclotomic polynomials, that is polynomials of the form $(x^n-1)/\Phi_n(x)$. It appears that their coefficients tend to be small in absolute value. These results can be extended to the Taylor expansion in $0$ of $1/\Phi_n$.\cite{mor}
\smallskip

This paper deals with modular inverses of cyclotomic polynomials. If $\Phi_m$
and $\Phi_n$ are coprime ({\it i.e.} $\gcd(\Phi_m, \Phi_n)=1$), then $\Phi_m$ is invertible modulo $\Phi_n$ and,
following the example of $\Phi_n$, we may ask whether the coefficients of
$\Phi_m ^{-1}\bmod \Phi_n$ are of special form. Especially, we noticed that
the magnitude of these coefficients is very particular when $n$ is composed of
few prime factors, and we thoroughly prove lower and upper bounds for them
when $m$ and $n$ are two distinct divisors of $pr$, the product of two primes.
In the case of the product of three primes $pqr$, the peculiar structure of
$\Phi_{pqr}$ may also yield interesting results, but this is out of the scope of
this work.

Our main motivation is the computation of a convenient morphism between the
multiplicative group of a finite field $\FF{q^n}$ and products of some of its
subgroups. 
Such calculations typically occur in torus-based cryptographic schemes, as developed by Silverberg and
Rubin~\cite{RS03,RS2}. The bounds presented in Theorem~\ref{th1} lead to improvements in the running times
of algorithms in this field (see~\cite{DW04,DL09}). 
Such schemes are discrete log-based cryptosystems and make use of a subgroup of $\FFx{q^n}$ in which the communication cost is reduced. That is to say elements can be represented by less than the usual $n$ coordinates in $\FF{q}$.

\smallskip

In Section~\ref{sec:algebraictori} we explain more precisely the geometric structure of algebraic tori, which is the mathematical context of torus-based cryptography. A cryptographic application of the results presented in this paper will be sketched in Section~\ref{sec:crypto-application}.

Section~\ref{sec:coprimality} is dedicated to the resultant of $\Phi_m$ and
$\Phi_n$. To that end, we show the following lemma.
\begin{lem}\label{lem1}
For all integers $m>n\geqs 1$,
\[\Res(\Phi_m,\Phi_n) \neq 1 \eq m=np^\alpha \quad \tx{with }p \tx{ prime and } \alpha\geqs1.
\]
\end{lem}
This is a consequence of a result by Apostol~\cite{apo} about the resultant of cyclotomic polynomials. We suspect that it is already known since it helps proving the equivalence of two definitions of $T_n(\FF{q})$ given in~\cite{RS03}, but we did not find any explicit proof of it in the literature. 
As a result, we obtain at the end of Section~\ref{sec:coprimality} a sufficient condition for the coprimality of $\Phi_m(q)$ and $\Phi_n(q)$, whatever the integer $q$ is.

In the case of two coprime cyclotomic polynomials, we can consider the
inverse of $\Phi_m$ modulo $\Phi_n$. In
Section~\ref{sec:inversion-phi_m-mod}, we make an exhaustive study when $n$
and $m$ are divisors of the product of two primes and we prove the following theorem.
\begin{thm}\label{th1}
For all $p$ and $r$ distinct prime numbers,
\begin{itemize}
\item[(i)] $\bo{{\Phi_{p}^{-1}} \bmod \Phi_1 = {1}/{p}}$ and
  $\bo{{\Phi_{1}^{-1}} \bmod \Phi_p = (-{1}/{p})(X^{p-2}+2X^{p-3} + \hdots
    +p-1)}$.
\item[(ii)] $\bo{{\Phi_{pr}^{-1}} \bmod \Phi_1 = 1}$ and $\bo{{\Phi_{1}^{-1}}
    \bmod \Phi_{pr} = \sum_{i=0}^{\phy(pr)-1} v_i X^i}$ with
  $v_i\in\{-1,0,1\}$.
\item[(iii)] $\bo{{\Phi_{pr}^{-1}} \bmod \Phi_p = \frac{1}{r}\sum_{i=0}^d
    X^i}$ with $d \equiv r-1 \bmod p$ and \\$\bo{{\Phi_{p}^{-1}} \bmod \Phi_{pr}
    = \frac{1}{r}\sum_{i=0}^{\phy(pr)-1} v_iX^i}$ with $v_i<r$.
\item[(iv)] $\bo{{\Phi_{p}^{-1}} \bmod \Phi_{r} = \sum_{i=0}^{\phy(r)-1}
    v_iX^i}$ with $v_i \in \{0, -1, +1\}$.
\end{itemize}
\end{thm}
\medskip

\noindent
\textit{Notations.}  In this paper, $\boP$ denotes the set of all prime numbers
and $(m,n)$ is the greatest common divisor of $m$ and $n$. We also recall the
following result about M\"obius $\mu$ function,
\begin{equation}
\forall n>1, \quad \bo{\sum_{d|n}\mu(d)=0}\,.\label{eq:1}
\end{equation}

\section{Geometry of algebraic tori}
\label{sec:algebraictori}

Many protocols and cryptosystems make use of the subgroup of order $\Phi_n(q)$ in the multiplicative group $\FFx{q^n}$. It is interesting to see it as the set of rational points over $\FF{q}$ of an algebraic torus. We refer to~\cite{COU,RS2} for more details.

\subsection{Structure of algebraic tori}

For a given field $K$, let $\bar{K}$ be a separable closure of $K$.
Let $\mG_m$ denote the multiplicative group. This is an affine absolutely connected algebraic group of dimension 1. 
An algebraic torus over $K$ is an algebraic group $T$ that is isomorphic to $\mG_m^s$ over $\bar K$, for some $s\geqslant 1$. We call \emph{splitting field of $T$} any subfield $L$ of $\bar K$ such that $T$ is isomorphic to $\mG_m^s$ over $L$.

From now on we consider finite extensions of finite fields. Let $L=\FF{q^n}$ be a field extension of $K=\FF{q}$ and let $G$ denote $\Gal(L/K)$. Let $\Res_{L/K}$ denote the functor of Weil restriction of scalars from $L$ to $K$. Its basic properties are given in \cite{VOS,WEI}. What we essentially need is that for a given variety $V$, there are $|G|$ functorial projection $L$-morphisms $\Res_{L/K} V\to V$ such that their direct sum gives an $L$-isomorphism
\[
\iota : \Res_{L/K} V \xrightarrow{\sim} V^{|G|}.
\]
In the case $V=\mG_m$, this isomorphism allows to represent an $L$-point of  $\Res_{L/K} \mG_m$ with $|G|$ coordinates taking values in $\mG_m\subset \mA^1$. We can define norm and trace maps by computing respectively the product and the sum of these coordinates.
Let $n=|G|$, we have the following explicit definition of the norm map. 

\[
\begin{array}{rccccl}\N_{L/K} : &\Res_{L/K} \mG_m  &\xrightarrow{\iota}& \mG_m^n &\rightarrow& \mG_m\\
&\alpha&\mapsto&(\alpha_g)_{g\in G}&\mapsto &\prod_{g\in G} \alpha_g,
\end{array}
\]
which happens to be defined over $K$.

More generally, for any intermediate extension $K\subseteq F\subseteq L$ we can construct partial norms $\N_{L/F,K} : \Res_{L/K}\mG_m\to \Res_{F/K}\mG_m$. These norms correspond to the usual ones on $L^\times$, the set of $K$-rational points of $\Res_{L/K}\mG_m$.

\begin{dfn}\label{def:normtorus}
The torus $T_L$ is defined as the intersection of the kernels of the norm maps $\N_{L/F,K}$ for all the subfields $K\subseteq F\subsetneq L$. 
\[
T_L = \bigcap_{K\subseteq F\subsetneq L}\Kern[\Res_{L/K} \mG_m\xrightarrow{\N_{L/F,K}} \Res_{F/K} \mG_m]
\]
\end{dfn}

With the usual norms over fields, we recover the elementary definition of the $K$-points of $T_L$:
\[
T_L(K) \simeq \left\{\alpha \in L^\times |\, \N_{L/F}(\alpha)=1 \quad \forall K\subset F\subsetneq L\right\}
\]


Moreover, this torus is $L$-isomorphic to $\mG_m^d$ with $d = \varphi(n)$. We refer to Proposition 2.6 of \cite{RS2} where Rubin and Silverberg give a detailed proof of this result.

\subsection{Endomorphisms of algebraic tori}
\label{subsec:endomorphisms}

Any algebraic torus $T$ of dimension $s$ is by definition isomorphic to $\mG_m^s$ over a splitting field. This means that it is actually a twist over $\FF{q}$ of $\mG_m^s$. So there exists a $\bar K$-isomorphism $I : T\to \mG_m^s$. 

We call $\sigma : \bar K \to \bar K$ the Frobenius automorphism. Let ${}^\sigma I : T\to \mG_m^s$ be the conjugate of $I$ by $\sigma$. The composition ${}^\sigma I I^{-1}$ is an endomorphism of $\mG_m^s$. Arguments in Galois cohomology~\cite{COU} show that there is a bijective correspondence which associates each twist of $\mG_m^s$ with the conjugacy classes of ${}^\sigma I I^{-1}$ inside the endomorphism ring of $\mG_m^s$. 

An endomorphism of $\mG_m^s$ is given by
\[
\mathfrak{a} : (g_1,\hdots, g_s) \mapsto \left(\prod_{1\leqslant j \leqslant s}g_j^{a_{i,j}}\right)_{1\leqslant i \leqslant s}.
\]

Such a map in characterized by the matrix of the exponents $(a_{i,j})_{1\leqslant i,j\leqslant s}$. This is a $s$-dimensional square matrix with integer coefficients, which actually corresponds to an endomorphism of the $\mZ$-module of characters of $\mG_m^s$.
The morphism $\mathfrak{a}$ is invertible if and only if the matrix $(a_{i,j})_{1\leqslant i,j\leqslant s}$ is invertible. So the automorphism group of $\mG_m^s$ is equal to $\GL_s(\mZ)$.


In the case of the Weil restriction $\Res_{\FF{q^n}/\FF{q}} \mG_m$ we obtain ${}^\sigma I I^{-1} = \omega$ where $\omega$ denotes the permutation of the coordinates, 
\[\omega(g_1,g_2,\hdots,g_n) = (g_n,g_1,\hdots,g_{n-1}).\]

Let us compute the ring of $\FF{q}$-endomorphisms of this torus.
With every endomorphism $\eps$ of $\mG_m^n$, we associate an endomorphism of $\Res_{\FF{q^n}/\FF{q}} \mG_m$ and the following diagram commutes. 
\[
\xymatrix{
\Res_{\FF{q^n}/\FF{q}} \mG_m \ar[d]_I^{\sim}\ar[rr]^{I^{-1}\varepsilon I}&& \Res_{\FF{q^n}/\FF{q}} \mG_m\ar[d]^{\sim}_I\\
\mG_m^n\ar[rr]^\varepsilon&&\mG_m^n 
}
\]

The endomorphism $I^{-1}\eps I$ is defined over $\FF{q}$ if and only if it is invariant under the action of $\sigma$, that is ${}^{\sigma}I^{-1}\varepsilon {}^{\sigma}I = I^{-1}\varepsilon I$.
So $\eps$ yields an $\FF{q}$-endomorphism of $\Res_{\FF{q^n}/\FF{q}} \mG_m$ if and only if $\omega\varepsilon = \varepsilon\omega$.



\subsection{Decomposition of $\mG_m^n$}

Paragraph~\ref{subsec:endomorphisms} shows that there is a functorial correspondence between the category of algebraic tori over finite fields and the category of $\mZ$-modules with an automorphism. 
For instance the torus $\Res_{\FF{q^d}/\FF{q}}\mG_m$ corresponds to $\mZ[X]/(X^q-1)$ with the automorphism $\omega$ given by the multiplication by $X$. 

The identity $X^n-1 = \prod_{d|n} \Phi_d(X)$ yields the isomorphism $\mQ[X]/(X^n-1)\simeq \prod_{d|n} \mQ[X]/\Phi_d(X)$. However we do not necessarily have an isomorphism between $\mZ[X]/(X^n-1)$ and  $\prod_{d|n} \mZ[X]/\Phi_d(X)$. Still we can write $(\mZ[1/n])[X]/(X^n-1) \simeq \prod_{d|n} (\mZ[1/n])[X]/\Phi_d(X)$. Consequently there exist two isogenies between the two algebraic groups $\Res_{\FF{q^n}/\FF{q}}\mG_m$ and $\prod_{d|n} T_{\FF{q^d}}$ such that their composition is the multiplication by a power of $n$.




Section \ref{sec:crypto-application} sketches how torus-based cryptography makes use of this decomposition up to isogeny. We will explain how the results in Sections~\ref{sec:coprimality} and~\ref{sec:inversion-phi_m-mod} on the coefficients of some specific polynomials allow to compute more efficiently these isogenies.
\section{Coprimality}
\label{sec:coprimality}

In this section, we consider the resultant of two cyclotomic polynomials
$\Phi_m$ and $\Phi_n$. In order to prove Lemma~\ref{lem1}, we start from
Theorem~\ref{th:apostol}, due to Apostol~\cite{apo}.

\begin{thm}[Apostol, \cite{apo}] \label{th:apostol} 
Let $m>1$, then
\[\Res(\Phi_1,\Phi_m)=
\begin{cases} 
p \quad \tx{if } m=p^a \tx{, } p \tx{ prime, } a\geqs 1, 
\\
1\quad \tx{otherwise}. 
\end{cases}
\]

Besides, let $m>n>1$, then
\begin{equation}
  \Res(\Phi_m,\Phi_n) = \bo{\prod_{\substack{d|n\\p\in \boP \tx{ such that }
        \frac{m}{(m,d)}=p^a}}p^{\mu (n/d)
      \frac{\phy(m)}{\phy(p^a)}}}\label{res}
\end{equation}
where $\mu$ is the M\"{o}bius function and $\phy$ the d'Euler totient
function. This product is performed over the divisors $d$ of $n$ such that
${m}/{(m,d)}$ is a prime power $p^a$.
\end{thm}

We can now prove Lemma~\ref{th1}.

\begin{proof} 
  For $n=1$, we have $\Res(\Phi_m,\Phi_1)=(-1)^{\phy(m)}\Res(\Phi_1,\Phi_m)$ and
  $\phy(m)$ is even as soon as $m>2$. So the resultant equals 1 if and only if
  $m$ is not a prime power.\smallskip

Now, let us consider $m>n>1$. This time we are going to use Eq.\eqref{res}.
\smallskip

\noindent
\textbf{Sufficiency.}
If $m= np^\alpha$, we can consider the powers of $p$ showing up in the product and show that it does not equal 1. 

For $d$ dividing $n$, we have ${m}/{(m,d)}=({n}/{d})p^\alpha$ because
$n|m$. So this is a power of $p$ only if $\bo{{n}/{d}}=p^\eps$. But
$\bo{\mu\left({n}/{d}\right)=0}$ as soon as $\eps>1$. So the only non trivial
terms will correspond to the cases $\eps=1$ and $\eps=0$.
\begin{itemize}
  \item For $\eps=0$, we have $d=n$, so $\bo{{m}/{(m,d)}=p^\alpha}$ and
  $a=\alpha$, which
  implies \[\bo{p^{\mu\left(\frac{n}{d}\right)\frac{\phy(m)}{\phy(p^a)}}=p^{\frac{\phy(m)}{\phy(p^\alpha)}}}. \]
  \item For $\eps=1$, we have $\bo{d={n}/{p}}$, so
  $\bo{{m}/{(m,d)}=p^{\alpha+1}}$ and $a=\alpha+1$, which implies
\[\bo{p^{\mu\left(\frac{n}{d}\right)\frac{\phy(m)}{\phy(p^a)}}=p^{-{\frac{\phy(m)}{\phy(p^{\alpha+1})}}}}.
\]
\end{itemize}
The contribution in terms of powers of $p$ is
${p^{\phy(m)\left(\frac{1}{\phy(p^\alpha)}-\frac{1}{\phy(p^{\alpha+1})}\right)}}
> 1.$
\smallskip

\noindent
\textbf{Necessity.}
We want to isolate the common factor of $m$ and $n$, if they have one. That is
to say, we write $m=wM \tx{ and }n=wN$ with $(M,N)=1$.

Since the resultant is not 1, we have at least one non trivial term $p^{\mu
  (n/d) \frac{\phy(m)}{\phy(p^a)}}$ in the product for some $d$ such that
$\bo{{m}/{(m,d)}}=p^a$.

Since $d|wN$, we can write $d=d_1d_2$ with $d_1|w$ and $d_2|N$. Let us suppose $d_1$ maximal for this property, {\it i.e.} $\bo{\left({w}/{d_1},d_2\right)=1}$. We call $\bo{w'={w}/{d}}$. Then
\begin{displaymath}
 {\frac{m}{(m,d)}} = \bo{\frac{wM}{(wM,d_1d_2)}}
= {\frac{wM}{d_1(w'M,d_2)}}
= {\frac{wM}{d_1(w',d_2)}}
= {{\frac{w}{(w,d)}} M}.
\end{displaymath}
So ${({w}/{(w,d)})M} = p^a$, which implies that $M$ is a power of $p$, say
$M=p^\alpha$. But $\bo{{w}/{(w,d)}}$ is also a power of $p$. Let $p^r$ be the
greatest power of $p$ in $w$, so that $w=p^rs$ with $p\nmid s$. Thus the
powers involved in the product
are ${\mu\left({n}/{d}\right)=\mu\left({p^rsN}/{d}\right)}$. 

We know that it will be zero as soon as $p$ has power at least 2 in $\mu$. So the
contribution to the product will be non trivial only if $d=p^{r-\eps}\delta$
with $\eps=0\tx{ or }1$ if $r\geqs 1$ and $\delta|sN$. More precisely
$\delta=s'd_2$ with $d_2|N$ and $s'|s$. Then $(w,d)=p^{r-\eps} (s/s')$ and we
can even assert that $s=s'$ because $\bo{{w}/{(w,d)}}$ must be a power of $p$
and $p\nmid s$. 

Finally,
\begin{displaymath}
w=p^r s,\ 
  d=p^{r-\eps}sd_2 \tx{ with } d_2|N \tx{ and } \eps=0 \tx{ or }1
\tx{ if } r\geqs 1.
\end{displaymath}
Thus $\bo{{w}/{(w,d)}=p^{a-\alpha}=p^\eps}$.
Now we can give the contribution to the product.

If $\eps=0$, then $d=wd_2$, $a=\alpha$ and the product is
\begin{displaymath}
\bo{\prod_{d|n}p^{\mu (n/d) \frac{\phy(m)}{\phy(p^a)}}} =\bo{\prod_{d_2|N}p^{\mu (N/d_2) \frac{\phy(m)}{\phy(p^\alpha)}}}=\bo{\left(p^{\frac{\phy(m)}{\phy(p^\alpha)}}\right)^{\bo{\sum_{\substack{d_2|N}}\mu\left(\frac{N}{d_2}\right)}}}
= \bo{\left(p^{\frac{\phy(m)}{\phy(p^\alpha)}}\right)^{\bo{\sum_{\substack{d_2|N}}\mu(d_2)}}}.
\end{displaymath}

If $\eps=1$, then $\bo{d={w}d_2/p}$, $a=\alpha+1$ and the product is
\begin{eqnarray*}
\bo{\prod_{d|n}p^{\mu (n/d) \frac{\phy(m)}{\phy(p^a)}}}&=&\bo{\prod_{d_2|N}p^{\mu \left(\frac{pN}{d_2}\right) \frac{\phy(m)}{\phy\left(p^{\alpha+1}\right)}}}
=\bo{\prod_{d_2|N}p^{\mu(p)\mu \left(\frac{N}{d_2}\right)
    \frac{\phy(m)}{\phy\left(p^{\alpha+1}\right)}}} \quad\tx{since } (p,N)=1,\\
&=&\bo{\left(p^{\mu(p)\frac{\phy(m)}{\phy\left(p^{\alpha+1}\right)}}\right)^{\bo{\sum_{\substack{d_2|N}}\mu \left(\frac{N}{d_2}\right)}}}
= \bo{\left(p^{\mu(p)\frac{\phy(m)}{\phy\left(p^{\alpha+1}\right)}}\right)^{\bo{\sum_{\substack{d_2|N}}\mu(d_2)}}} \,.
\end{eqnarray*}
Thus,
\begin{displaymath}
  {\left(p^{\frac{\phy(m)}{\phy(p^\alpha)}}p^{\mu(p)\frac{\phy(m)}{\phy\left(p^{\alpha+1}\right)}}\right)^{\bo{\sum_{\substack{d_2|N}}\mu(d_2)}}}\text{
    divides }\Res(\Phi_m,\Phi_n)\,.
\end{displaymath}

Since this result should be greater than 1, necessarily
$\bo{\sum_{\substack{d_2|N}}\mu(d_2)}\neq 0$ which is impossible unless $N=1$
according to Eq.~\eqref{eq:1}, and thus
$n=w$ and $m=np^\alpha$.
\end{proof}

Now it is easy to show the following condition of coprimality.

\begin{cor}
  For any integer $q$ and $m>n\geqs 1$ integers, $\Phi_m(q)$ and $\Phi_n(q)$
  are coprime if $m$ does not divide $n$.
\end{cor}

\begin{proof}
  If $m$ does not divide $n$, we know from Lemma~\ref{lem1} that
  $\Res(\Phi_m,\Phi_n) = 1$, which is true in $\mZ$ but also in $\mZ/\ell\mZ$
  for any $\ell\in \mZ$ since 1 is unchanged.  Now suppose that $\Phi_m(q)$
  and $\Phi_n(q)$ have a common factor, say $\ell$.  Then $\Phi_m$ and
  $\Phi_n$ have a common root, $q$, in $\mZ/\ell\mZ$ and consequently their
  resultant is zero, which is false.
\end{proof}

\section{Inversion of $\Phi_m$ mod $\Phi_n$}
\label{sec:inversion-phi_m-mod}

Consider $m$ and $n$ such that the cyclotomic polynomials $\Phi_m$ and $\Phi_n$ are coprime.  Then $\Phi_m$ is invertible modulo $\Phi_n$ and it is a natural question to try to compute $\bo{{\Phi_m}^{-1}}$ modulo $\Phi_n$ and more precisely we would like to know the magnitude of its coefficients.

Since $\Phi_m$ and $\Phi_n$ are coprime we can write the B\'ezout relation  
\begin{equation}\label{bezout}
\Phi_m U+\Phi_n V = 1.
\end{equation} 

Our goal is to study $U = {\Phi_m}^{-1} \bmod \Phi_n$.

In this section we are going to prove the four assertions of Theorem~\ref{th1} in turn. We recall that $\Phi_n(1) = p$ if $n=p^{\alpha}$ is a prime power; else $\Phi_n(1) = 0$ for $n>1$.

\subsection{Case $m=p$ and $n=1$.}

The cyclotomic polynomials $\Phi_p$ and $\Phi_1$ are both easy to write and it is not difficult to obtain explicit expressions for their inverses.

\begin{prop}
For all prime number $p$,
\begin{itemize}
\item $\bo{{\Phi_{p}^{-1}} \bmod \Phi_1 = {1}/{p}}$,
\item $\bo{{\Phi_{1}^{-1}} \bmod \Phi_p = -\frac{1}{p}(X^{p-2}+2X^{p-3} + \hdots +p-1)}$. 
\end{itemize}
\end{prop}

\begin{proof}
We simply check that the B\'ezout relation between $\Phi_{p}$ and $\Phi_{1}$ is valid.
\[
\begin{split}
-\Phi_1(X)(X^{p-2}+2X^{p-3} &+ \hdots +p-1) + \Phi_p(X)\\& = (X-1)\sum_{k=0}^{p-2} (k+1-p)X^k+ \Phi_p(X)\\
&=\sum_{k=1}^{p-1} (k-p)X^{k} -\sum_{k=0}^{p-2} (k+1-p)X^k+\sum_{k=0}^{p-1} X^k =p. 
\end{split}
\]
\end{proof}

\subsection{Case $m=pr$ and $n=1$.}

The explicit expression of $\Phi_{pr}$ is less convenient than that of $\Phi_p$, but we still have useful information thanks to Lam and Leung \cite{lam}.

\begin{prop}\label{prop2}
For all $p$ and $r$ distinct prime numbers, 
\begin{itemize}
\item $\bo{{\Phi_{pr}^{-1}} \bmod \Phi_1 = 1}$.
\item $\bo{{\Phi_{1}^{-1}} \bmod \Phi_{pr} = \sum_{i=0}^{(p-1)(r-1)-1} v_i X^i}$ with $v_i\in\{-1,0,1\}$.
\end{itemize}
\end{prop}

\begin{proof}

We are first looking for $U$ in the B\'ezout relation $\Phi_{pr} U +\Phi_1V = 1$ and we know that it has degree 0. So a simple evaluation of this relation at 1 gives $U(1)=1$, so ${\Phi_{pr}}^{-1} \bmod \Phi_1 = 1$ because $\Phi_{pr}(1)=1$.

Now $V$ is characterized by $(X-1)V(X) = 1-\Phi_{pr}(X)$. Let $V(X)=\sum_{i=0}^d v_i X^i$ and $\Phi_{pr}(X) =\sum_{i=0}^d a_i X^i$ with $d= (p-1)(r-1)-1$. 
Then we can write the equation as a linear system,
\[
\begin{bmatrix}-1&0&\hdots&0\\1&-1&\ddots&\vdots\\ \vdots&\ddots&\ddots&0\\0&\hdots&1&-1\end{bmatrix}\begin{bmatrix}v_0\\v_1\\\vdots\\v_d\end{bmatrix}=\begin{bmatrix}1-a_0\\-a_1\\\vdots\\-a_d\end{bmatrix}\, \Leftrightarrow\, \begin{bmatrix}v_0\\v_1\\\vdots\\v_d\end{bmatrix} = \begin{bmatrix}1&0&\ddots&0\\1&1&\ddots&\vdots\\\vdots&\ddots&\ddots&0\\1&\hdots&1&1\end{bmatrix}\begin{bmatrix}a_0-1\\a_1\\\vdots\\ta_d\end{bmatrix}.
\]
Note that we know from \cite{lam} that $a_0 =1$ and for all $i$, $a_i\in\{0,\pm1\}$. Moreover the signs ($+1$ or $-1$) are alternating. So each $v_i$ is necessarily 0 or $\pm 1$.
\end{proof}

A similar technique could allow us to compute $\Phi_n^{-1}$ modulo $\Phi_1$ for any $n$ since it is simply $\Phi_n(1)$ which we explicitly know. 

\subsection{Case $m=pr$ and $n=p$.}

This time we will need the explicit expression of $\Phi_{pr}$.

\begin{prop}\label{prop3}
For all $p$ and $r$ distinct prime numbers,
\[\bo{{\Phi_{pr}^{-1}} \bmod \Phi_p = \frac{1}{r}\sum_{i=0}^d X^i} \tx{ with } d \equiv r-1 \bmod p.\]
\end{prop}

\begin{proof}
Let us directly show that $\bo{\frac{1}{r} \left(\sum_{i=0}^{d} X^i \right) \Phi_{pr}} \equiv 1 \bmod \Phi_p.$ For this purpose, we need to use the expression of $\Phi_{pr}$ given in \cite{lam}. Let $s$ and $t$ be two integers such that $(p-1)(r-1) = \phy (pr) = sp+tr$. Then, 
\[
\bo{\Phi_{pr}(X) = \left(\sum_{i=0}^{s}X^{ip}\right)\left(\sum_{j=0}^{t}X^{jr}\right) -  \left(\sum_{i=s+1}^{r-1}X^{ip}\right)\left(\sum_{j=t+1}^{p-1}X^{jr}\right)X^{-pr}}. 
\] 
Thus,
\[
\begin{split}
\Phi_{pr}\sum_{i=0}^{d} X^i \bmod \Phi_p 
&=\left(
(s+1)\sum_{j=0}^{t}X^{jr} -  (r-1-s)\sum_{j=t+1}^{p-1}X^{jr}\right)\sum_{i=0}^{d} X^i 
\\&= \left((s+1)\sum_{j=0}^{p-1}X^{jr} -  r\sum_{j=t+1}^{p-1}X^{jr}\right)\sum_{i=0}^{d} X^i 
\end{split}\]
But $\Phi_p(X^r)=\sum_{j=0}^{p-1} X^{jr}$, and then
\[
\begin{split}
\Phi_{pr}\sum_{i=0}^{d} X^i \bmod \Phi_p &= 
-  r\sum_{j=t+1}^{p-1}X^{jr}\sum_{i=0}^{d} X^i \\
&=r\frac{X^{(t+1)r}-1}{X^r-1}\,\frac{X^{d+1}-1}{X-1}
\end{split}
\]

An explicit computation shows that $(t+1)r = 1+pr-p(s+1)$. So $X^{(t+1)r}\equiv X \bmod \Phi_p$.  Besides $d+1 \equiv r \bmod p$, so $X^{d+1}\equiv X^r \bmod \Phi_p$, which leads to the result. There only remains $r$ in the computed product. 
\end{proof}

\subsection{Case $m=p$ and $n=pr$.} 

\begin{prop}
For all $p$ and $r$ distinct prime numbers,
\[\bo{{\Phi_{p}^{-1}} \bmod \Phi_{pr} = \frac{1}{r}\sum_{i} v_iX^i} \tx{ with } v_i<r.\]
\end{prop}

\begin{proof}
We are looking for $V$ in the B\'ezout relation $\Phi_{pr}U+\Phi_p V = 1$, where $U=\frac{1}{r}\sum_{i=0}^d X^i$ from Proposition \ref{prop3}. So, all we have to do is divide $1-\Phi_{pr} U$ by $\Phi_p$. 

First, note that $1-\Phi_{pr}U$ has only coefficients $\pm 1$. Indeed we know that  the coefficients of $\Phi_{pr}$ are alternating +1's and -1's among other 0's. So, if we write their explicit product, we obtain a polynomial with simple coefficients (only 0 or $\pm1$) thanks to the Cauchy product. In fact, we divide a polynomial with coefficients in $\{0, \pm 1\}$ by $\Phi_p$. 

If we simply examine the Euclidean division step by step, we can show by recurrence that the range of possible coefficients for the quotient increases by one at each step, and always contains~0. 
\end{proof}

\subsection{Case $m=p$ and $n=r$ with $p$ and $r$ two distinct primes}

Before we give a proof for the last assertion in Theorem~\ref{th1}, we need to work on the general problem. 
The idea is to evaluate our B\'ezout relation at the roots of $\Phi_n$ and to interpolate $U$ from the values found at these points. Yet we are going to slightly modify the equation in order to have a more convenient linear system.

Recall that $\Phi_p U+\Phi_r V = 1$ (Eq. \ref{bezout}). If we multiply both sides by $X-1$, we obtain $\Phi_p \widetilde{U}+(X^r-1) V = X-1$, with $\widetilde{U} = (X-1)U$. 

The roots of $X^r - 1$ are the $r$-th  roots of 1, so $\{\xi^j, 0\leqs j\leqs r-1\}$.
The evaluation of our B\'ezout relation at these points gives 
\[\forall 0\leqs j\leqs r-1, \quad \Phi_p(\xi^j)\widetilde{U}(\xi^j) = \xi^j-1.
\]
If we note $\widetilde{U} = \sum_{i=1}^{r} \widetilde u_i X^{i-1}$ then the equation can be written \[
\forall 0\leqs j\leqs r-1, \quad \bo{ \sum_{i=1}^{r} \widetilde u_i (\xi^j)^{i-1} = ( \xi^j-1)(\Phi_p ( \xi^j))^{-1}}.
\]

We first work on $\widetilde U$ and its coefficients.

\begin{lem}\label{utilde}
\[\forall 1\leqs i \leqs r, \quad \widetilde u_i \in \{0, +1, -1\}.
\]
\end{lem}

\begin{proof}
The coefficients $(\widetilde u_j)_{1\leqs j\leqs r-1}$ are the solutions of a system of linear equations whose matricial version is $A \widetilde U = W$ where $W = \begin{bmatrix} (\xi^{j}-1)\Phi_p ( \xi^{j})^{-1} \end{bmatrix}_{0\leqs j\leqs r-1}$ and
\[
A = \begin{bmatrix}
1&1 & \hdots & 1\\
1&\xi&\hdots&\xi^{r-1}\\
\vdots&\vdots &\vdots&\vdots\\
1&\xi^{r-1}&\hdots&(\xi^{r-1})^{r-1}\\
\end{bmatrix}=\mathrm{VdM}(1,\xi,\hdots , \xi^{r-1}).
\]

Here VdM denotes the Vandermonde matrix. $A$ is invertible since all $(\xi^i)_{i\in \{0,\hdots,r-1\}}$ are distinct. Thus we can give an explicit resolution of the system : $U=A^{-1}W$. It is proven in \cite{mca} that the inverse of a Vandermonde matrix is still a Vandermonde matrix with inverse coefficients. Here $A^{-1}=(1/r) \mathrm{VdM}(1,\xi^{-1},\hdots,\xi^{-r+1})$ {\it i.e.} its explicit coefficients are $A^{-1}=1/r[(\xi^{-(i-1)})^{j-1}]_{\substack{1\leqs i\leqs r\\0\leqs j \leqs r-1}}$.

\smallskip
So the solutions of the linear system are given by
\[
\forall 1\leqs i \leqs r,\quad
\widetilde u_i = \frac{1}{r}\sum_{j=0}^{r-1} (\xi^{-(i-1)})^{j}(\xi^{j}-1)\Phi_p ( \xi^{j})^{-1}. 
\]
Now using $\Phi_p(X) = (1-X^p)/(1-X)$, we find
\[
\forall 1\leqs i \leqs r,\quad
\widetilde u_i = \frac{1}{r}\sum_{j=0}^{r-1} (\xi^{1-i})^{j}(\xi^{j}-1)\frac{1-\xi^j}{1-\xi^{jp}}.
\]

We can improve this expression using the following relation: 
\[
\bo{\frac{1}{1-\xi^{jp}}= \frac{1}{r} \left(\xi^{jp(r-2)}+ 2\xi^{jp(r-3)} +\hdots +(r-1)\right)}.
\]

Indeed, 
\begin{eqnarray*}
\bo{(1-\xi^{jp}) \sum_{k=0}^{r-2} (r-1-k)\xi^{jpk}}&=&\bo{\sum_{k=0}^{r-2} (r-1-k)\xi^{jpk} - \sum_{k=0}^{r-2} (r-1-k)\xi^{jp(k+1)}}\\
&=&\bo{\sum_{k=0}^{r-2} (r-1-k)\xi^{jpk} - \sum_{k=1}^{r-1} (r-k)\xi^{jpk}}\\
&=&\bo{(r-1) - \underbrace{\sum_{k=1}^{r-2}(\xi^{jp})^{k}}_{=-1-\xi^{jp(r-1)} \tx{ if } jp \nequiv 0 \bmod r} -\xi^{jp(r-1)}}\\&=& r \quad\tx{since } jp \nequiv 0 \bmod r \tx{ ($p$ prime and $1\leqs j\leqs r-1$).}
\end{eqnarray*}

So the final expression for all $1\leqs i \leqs r$ is
\[
\widetilde u_i = \frac{1}{r^2}\sum_{k=0}^{r-2}(r-k-1)\sum_{j=0}^{r-1} \xi^{j(1-i)}(\xi^{j}-1)(1-\xi^j)\xi^{jpk}.
\]

After developing and collecting, we will work on the following form.
\[
\widetilde u_i = -\frac{1}{r^2}\sum_{k=0}^{r-2}(r-k-1)\sum_{j=0}^{r-1} (\xi^{j(pk+1-i)} - 2  \xi^{j(pk+2-i)}+ \xi^{j(pk+3-i)}),\]
\[
\widetilde u_i = -\frac{1}{r^2}\sum_{k=0}^{r-2}(r-k-1)\underbrace{ (S_1(k)-2S_2(k)+S_3(k))}_{S(k)},
\]
where $S_{l}(k) = \sum_{j=0}^{r-1} \xi^{j(pk+l-i)}$.

The sums $S_l =\sum_{j=0}^{r-1} (\xi^A)^j$ are actually sums of all the powers of a $r$-th root of 1. So if $\xi^A$ is a prime root of 1, the sum simply equals 0. And if $\xi^A$ is not a prime root of 1, the only possibility is $\xi^A=1$ ({\it i.e.} $A\equiv 0 \bmod r$) and in this case the sum equals $r$.

So $\widetilde u_i$ depends on the value of the powers $pk+l-i$ modulo $q$.

Most of time the three sums involved will all be equal to 0 and thus $S_1-2S_2+S_3 = 0$. But there can be up to three values of $k$ for which one of the sums will not be equal to 1 but to $r$.

\begin{itemize}
\item If  there exists $0\leqs k_1\leqs r-2$ such that $pk_1+1-i\equiv0 \bmod r$ then $S_1(k_1) = r$.
\item If  there exists $0\leqs k_2\leqs r-2$ such that $pk_2+2-i\equiv0 \bmod r$ then $S_2(k_2) = r$.
\item If  there exists $0\leqs k_3\leqs r-2$ such that $pk_3+3-i\equiv0 \bmod r$ then $S_3(k_3) = r$.
\end{itemize}

The most important argument now is the following : in each sum $S_l$ all the powers of $\xi$ appear, except the power involving $k=r-1$. Besides the powers $\{pk+l-1 \bmod r, 0\leqs k \leqs r-1\}$ take all values $\{0,\hdots,r-1\}$ because $p$ and $r$ are coprime. So either there will exist  $0\leqs k_1\leqs r-2$ such that $1+pk_1\equiv0 \bmod r$ or necessarily $p(r-1)+l-i\equiv 0 \bmod r$.

Let us first prove that at least two $k_l$'s among $k_1$, $k_2$ and $k_3$ exist. If there does not exist a $k_l$, $l=1, 2 \tx{ or } 3$, then   $p(r-1)+l-i \equiv 0 [r]$. So if two different $k$'s do not exist, we will have this relation for $l$ and $l'$ among $\{1,2,3\}$. So $l\equiv l' \bmod r$, which is impossible if $l\neq l'$. This proves that at least two of the three $k_l$'s exist. 

The potential nonzero contributions to $\widetilde u_i$  are $S(k_1) = r$, $-2S(k_2) = -2r$ and $S(k_3) =  r$ when respectively $k_1$, $k_2$ and $k_3$ as above exist.

As we have previously seen, at most one $k_l$ does not exist. So we have four different cases.

\begin{itemize}

\item If $k_1,k_2$ and $k_3$ in $\{0,\hdots,r-2\}$ exist, then we have the system of three equations
 \[\left\{
 \begin{array}{rclc}
pk_1+1-i&\equiv& 0 [r]&(a)\\pk_2+2-i&\equiv& 0 [r] & (b)\\pk_3+3-i&\equiv &0 [r] &(c)
 \end{array}\right. .\]
 
 So $p(k_1+k_3-2k_2) \equiv 0 \bmod r$.
Since $p$ and $r$ are coprime, $r|(k_1+k_3-2k_2)$. Since no $|k_l|$ exceeds $r-1$, $k_1+k_3-2k_2- = 0$ or $\pm r$.

So $-r^2\widetilde u_i = (r-k_1-1)r+(r-k_2-1)(-2r)+(r-k_3-1)r = 0$ or $\pm r^2.$ Thus $ \widetilde u_i = 0$ or $\pm 1$.

\item If $k_1$ does not exist then Eq. $(a)$ is replaced by $p(r-1)+i-1 \equiv 0 \bmod r$. So $p\equiv 1-i \bmod r$. Then plugging  this particular $p$ in equations $(b)$ and $(c)$ gives $(1-i)(k_2+1)\equiv -1 \bmod r$ and $(1-i)(k_3+1) \equiv -2\bmod r$. Hence the relation $2k_2-k_3+1 =0$.

Then $-r^2(1-i)\widetilde u_i = (1-i)(r-k_2-1)(-2r)+(1-i)(r-k_3-1)r = -(1-i)r^2$.
But $1-i \neq 0$ (or else $p\equiv0 \bmod r$, which is impossible).
So  $\widetilde u_i=1$.

\item If $k_2$ does not exist then Eq. $(b)$ is replaced by $p(r-1)+2-i \equiv 0 \bmod r$. So $p\equiv 2-i \bmod r$; in particular $i\neq 2$. Then with  this value of $p$, equations $(a)$ and $(c)$ give, similarly to the previous case, $(2-i)(k_1+1)\equiv 1 \bmod r$ and $(2-i)(k_3+1) \equiv -1\bmod r$. Hence the relation $k_1+k_3+2 =r$.

Then $-r^2(2-i)\widetilde u_i = (2-i)(r-k_1-1)r+(2-i)(r-k_3-1)r = (2-i)r^2$.
So  $\widetilde u_i=-1$.

\item If $k_3$ does not exist then Eq. $(c)$ is replaced by $p(r-1)+3-i \equiv 0 \bmod r$. So $p\equiv 3-i \bmod r$ and $i\neq 3$. Then with  this particular $p$, solving equations $(a)$ and $(c)$ gives similarly to previous cases $2k_2-k_1+1 =r$.

Then $-r^2(3-i)\widetilde u_i =  0$. So  $\widetilde u_i=0$.
\end{itemize}
\end{proof}

\begin{prop} For all distinct prime numbers $p$ and $r$,  \[\Phi_{p}^{-1} \bmod \Phi_{r} = \sum_{i=1}^{\phy(r)} u_iX^{i-1} \tx{ with } u_i \in \{ -1, 0, +1\}.\]
\end{prop}

\begin{proof}

Now we can compute the coefficients of $U$ such that $\widetilde U = (X-1) U$. A similar calculation has been performed for the proof of Proposition \ref{prop2}. Indeed it has the following matricial formulation:
\[
\begin{bmatrix}-1&0&\hdots&0\\1&-1&\ddots&\vdots\\ \vdots&\ddots&\ddots&0\\0&\hdots&1&-1\end{bmatrix}\begin{bmatrix}u_1\\u_2\\\vdots\\u_{r}\end{bmatrix}=\begin{bmatrix}\widetilde u_1\\\widetilde u_2\\\vdots\\\widetilde u_{r}\end{bmatrix}, \tx{ equivalent to} \begin{bmatrix}u_1\\u_2\\\vdots\\u_{r}\end{bmatrix} = \begin{bmatrix}1&0&\hdots&0\\1&1&\ddots&\vdots\\\vdots&\ddots&\ddots&0\\1&\hdots&1&1\end{bmatrix}\begin{bmatrix}\widetilde u_1\\\widetilde u_2\\\vdots\\\widetilde u_{r}\end{bmatrix}.
\]

Since $u_i$ is a sum of consecutive $\widetilde u_j$'s, all we need to prove is that the +1's and the -1's alternate in $(\widetilde u_j)_{1\leqs j \leqs r}$ (among possible zeros). 

For each $1\leqs i \leqs r$ and $l \in \{1,2,3\}$, recalling the notations above, put $K(i) = (k_1(i),k_2(i),k_3(i))$ where $k_l(i)\in \{0,\hdots, r-1\}$ is the coefficient such that $pk_l(i)+i-1\equiv 0 \bmod r$. Then $r\widetilde u_i = k_1(i)-2k_2(i) +k_3(i)$.

We know that $i \equiv l+pk_l(i) [r]$. So $k_3(i)=k_2(i)-1/p = k_1(i)-2/p$.
Knowing $K(i)$ also allows us to find $K(i+1)$. 
Indeed $i+1\equiv 1+l+pk_l(i+1)[r]$. Thus 
\[
\begin{cases}k_2(i+1) = k_1(i)\\k_3(i+1) = k_2(i)\\k_1(i+1) \equiv k_3(i) +3/p\bmod r
\end{cases}.
\] 

Finally, given $k=k_1(i)$, then $K(i) = (k, k-1/p, k-2/p)$ and the next one is $K(i+1)=(k+1/p, k, k-1/p)$ (all the values taken modulo r). 

Now we are able to describe whether $r\widetilde u_i = k_1(i)-2k_2(i) +k_3(i)$ equals $-r, 0$ or $r$. The rotation of these different values of $r\widetilde u_i$ depends on the rotation of the $k_l$'s modulo $r$. If we compute the $r\widetilde u_i$'s successively, the $k_l$'s involved increase by $1/p$ at each step. When $k_1$ or $k_3$ should exceed $r$, taking its value modulo $r$ results in a drop by $r$ in $r\widetilde u_i$. Similarly if $k_2$ should exceed $r$, taking its value modulo $r$ results in adding $2r$ to $r\widetilde u_i$ (since $k_2$ counts with coefficient -2). Since the $k_l$'s will alternatively exceed $r$ (always $k_1$ after $k_2$ after $k_3$ after $k_1$...), we will alternatively add $-r$, $-r$ and $+2r$ to $r\widetilde u_i$ when computing $r\widetilde u_{i+1}$. The number of such operations at each step depends on $p$ and $i$ but will not exceed 2 (the three $k_l$'s can't all exceed $r$ at the same step since we add $1/p\bmod r$ and their range is twice this value). So checking whether the first values belong to $[-r, r]$ suffices to prove that it will be so for the rest of the coefficients by iteration of the process.

The first coefficient ($i=1$) has $K(1) = (0, -1/p \bmod r, -2/p \bmod r)$. So there are two possible sets. 
First if $1/p<r/2$, then $K(1) = (0, r-1/p, r-2/p)$ and then $r\widetilde u_1 = -r$. Since $k_2(1)>k_1(1)$, we start with the addition of $2r$. Indeed $k_2(2) = r \bmod r =0$, which corresponds to increasing $r\widetilde u_1$ by $2r$: $r\widetilde u_2=r$. So the initiation of the process is correct. Else we have $1/p>r/2$ and so $K(1) = (0, r-1/p, 2r-2/p)$. Then $r\widetilde u_1=0$. Similarly the next set will be $K(2) = (1/p, 0, r-1/p)$, adding both $-r$ and $2r$ to $r\widetilde u_1$, which yields $r\widetilde u_2=r$. The initiation is correct too.

This completes the proof since the alternance of +1's and -1's in $(\widetilde u_i)_{1\leqs i \leqs r}$ shows that $u_i \in \{-1, 0, 1\}$ for all $1\leqs i \leqs r$.

\end{proof}

\section{A cryptographic application}
\label{sec:crypto-application}

Beyond the simple arithmetic context of this computation, we actually found a direct application in torus-based cryptography. In this section we will briefly present the use of Theorem~\ref{th1} that has been made in \cite{DL09}. We refer to the latter for more details and more references. 

During the past twenty years, practical torus-based cryptosystems have been constructed for different extension degrees such as 2, 3 or 6 (see for instance {\sc luc}\cite{SL93}, {\sc xtr}\cite{LV00} or {\sc ceileidh}\cite{RS03}). Yet the search for rational parametrizations of algebraic tori has raised several unsolved questions. 
Following the ideas of van Dijk and Woodruff~\cite{DW04}, we construct a map $\theta$ whose kernel is annihilated by a power of $n$, so that $\theta$ is not far from being a bijection.

\begin{equation}
  \label{eq:2}
  \theta: \TT{n} \times
  \prod_{\substack{d\,|\,n\\ \mu(n/d)=-1}} \FFx{q^d} \to
  \prod_{\substack{d\,|\,n\\ \mu(n/d)=+1}} \FFx{q^d}\,.
\end{equation}

This kind of parametrization notably finds applications in such cryptosystems as Diffie-Hellman multiple key exchange. In \cite{DL09} we present a practical implementation of this map, whose efficiency relies on the use of a certain class of normal bases (see \cite{CL09}) in the representation of field extensions. 

We suppose that the dimension $n$ is the product of two distinct primes $p$ and $r$, and we now give explicit details on the computation of $\theta$.

In the sequel we are going to use several times the following principle. Given the resultant of two polynomials $P$ and $Q$, we know that there exist $U$ and $V$ such that
\[
U(X)P(X)+V(X)Q(X) = \Res(P,Q).
\]
Evaluating this relation at some integer yields a B\'ezout-like relation showing that $\pgcd(P(q),Q(q))$ divides $\Res(P,Q)$.
In particular if we use Theorem~\ref{th:apostol}, we have a relation between the evaluations of two cyclotomic polynomials.
\[
U(q) \Phi_n(q) + V(q) \Phi_m(q) = \Res(\Phi_n,\Phi_m).
\]

Let us first consider the simple example of $\FFx{q^p}$. Let $T_1$ and $T_p$ denote its subgroups of order $q-1$ and $\Phi_p(q)$ respectively. 
Then we have the two following norm maps.
\[\begin{array}{rclcrcl} \FFx{q^p}& \to &T_1& \textrm{ and } &\FFx{q^p}&\to&T_p\\
x_p&\mapsto&x_p^{\Phi_p(q)}&&x_p&\mapsto&x_p^{q-1}.
\end{array}\]
Besides, since $\Res(\Phi_1, \Phi_p) = p$, we can obtain an equation linking $q-1$ and $\Phi_p(q)$,
\[
\Phi_p(q)u_1+(q-1)u_p = p.
\]
with $u_1$ and $u_p$ integers. Thus we also have the following reverse map
\[
\begin{array}{ccl}
T_1\times T_p&\to&\FFx{q^p}\\(t_1,t_p)&\mapsto&t_1^{u_1}t_p^{u_p}.
\end{array}
\]
It is such that its composition with the product of the two norm maps above results in the multiplication by $p$.

We have a similar construction for $\FFx{p^r}$, writing $T_r$ its subgroup of order $\Phi_r(q)$.
\[
\begin{array}{rcc}
 \FFx{q^r}& \to &T_1\times T_r\\
 x_r&\mapsto & (x_r^{\Phi_r(q)},x_r^{q-1})\\
 t_1^{v_1}t_r^{v_r}&\mapsfrom & (t_1,t_r)
\end{array}
\]
with the relation $\Phi_r(q)v_1+(q-1)v_r = r$.

Now in the case of $\FFx{q^{pr}}$ we consider the four subgroups of order $q-1$, $\Phi_p(q)$, $\Phi_r(q)$ and $\Phi_{pr}(q)$ which we call $T_1$, $T_{p}$, $T_{r}$ and $T_{pr}$ respectively. 
Of course $T_1=\FFx{q}$, $T_p\subset\FFx{q^p}$ and $T_r\subset\FFx{q^r}$.


We have the following map whose components are the four natural norms.
\[
\begin{array}{rcc}
 \FFx{q^{pr}}& \to &T_1\times T_p\times T_r\times T_{pr}\\
 x_{pr}&\mapsto & (x_{pr}^{U_1(q)},x_{pr}^{U_p(q)},x_{pr}^{U_r(q)},x_{pr}^{U_{pr}(q)}),\\
\end{array}
\]
where $U_k(X) = {X^{pr}-1\over \Phi_k(X)}$.

Now we look for an inverse of this map. Following the previous example, for any B\'ezout-like relation,
\[
U_1V_1+U_{p}V_{p}+U_rV_r+U_{pr}V_{pr} = pr,
\]
we can construct a map
\[
\begin{array}{ccl}
 T_1\times T_p\times T_r\times T_{pr}& \to &\FFx{q^{pr}}\\
 (t_1,t_p,t_r,t_{pr})&\mapsto & t_1^{V_1(q)}t_{p}^{V_p(q)}t_{r}^{V_r(q)}t_{pr}^{V_{pr}(q)}.\\
\end{array}
\]
It is such that the composition of both maps yields the multiplication by $pr$ on $\FFx{q^{pr}}$.



In practice, we obtain such a relation in two steps. First we write two B\'ezout relations, between $\Phi_{pr}$ and $\Phi_1$ on the one hand and between $\Phi_{p}$ and $\Phi_r$ on the other hand. So the first step consists in two mappings,
\begin{displaymath}
\begin{array}[hb]{ccl}
  T_{1}\times T_{pr} &\xrightarrow{\sim}& G_1 \subset \FFx{q^{pr}}\,\\
  (t_1,t_{pr})&\mapsto & y_1 = t_1^{u_1}t_{pr}^{u_{pr}}\,,\end{array}\text{ where }
  \Phi_{pr}(q)u_1+\Phi_1(q)u_{pr} = 1\,
\end{displaymath}
and
\begin{displaymath}
  \begin{array}[h]{ccl}T_{p}\times T_{r} &\xrightarrow{\sim}& G_2 \subset \FFx{q^{pr}}\\
    (t_p,t_{r})&\mapsto & y_2 = t_p^{u_p}t_{r}^{u_{r}},\end{array}\text{ where }
  \Phi_{r}(q)u_p+\Phi_p(q)u_{r} = 1\,.
\end{displaymath}

Then we write a B\'ezout-like relation linking $\Phi_p\Phi_r$ and $\Phi_1\Phi_{pr}$. Theorem~\ref{th1} ensures that $(\Phi_p\Phi_r)^{-1}$ yields a factor $1/{pr}$ both modulo $\Phi_1$ and $\Phi_{pr}$. After recombination, this results in the following relation: there exist polynomials $V_1$ and $V_2$ with integer coefficients such that
\[
(\Phi_p\Phi_r) V_1 + (\Phi_1\Phi_{pr}) V_2 = pr.
\]
Thus, we combine the images $y_1\in G_1$ and $y_2\in G_2$ to form the element of $\FF{q^{pr}}$.  
\begin{displaymath}
  \begin{array}{ccl} G_1\times G_2&\to & \FFx{q^{pr}} \\
    (y_1,y_2) &\mapsto& y_1^{V_1(q)}y_2^{V_2(q)} \end{array}
\end{displaymath}
We set $v_1=V_1(q)$ and $v_2=V_2(q)$ and we summarize this process in Figure~\ref{fig:reconstruct}.

\begin{figure}[ht]
  \centering $\xymatrix @!0 @R=0.6cm @C=16mm {
    (\TT{1}\times \TT{{pr}})&\times &(\TT{p}\times \TT{r})\ar[rr]&&\FFx{q^{pr}}\\
    (t_1,t_{pr})\ar@{|->}@/_25pt/[dd]&,&(t_p,t_{r})\ar@{|->}@/^25pt/[dd]&&x_{pr} = y_1^{v_1}y_2^{v_2}\\
    G_1&\times&G_2\\y_1 = t_1^{u_1}t_{pr}^{u_{pr}}&&y_2
    =t_p^{u_p}t_{r}^{u_{r}}\ar@{|->}@/_15pt/[uurr] }$
  \caption{Reconstruction step in the case $n=pr$.}
  \label{fig:reconstruct}
\end{figure}

All in all, composing the different decompositions and recombinations presented here, we manage to give an explicit way of computing the map $\theta$ (see Figure~\ref{fig:mapping}).


\begin{figure}[h]
\begin{displaymath}
  \xymatrix @!0 @R=1cm @C=11.1mm {
    T_{pr}(\FF{q})&\times&\FFx{q^p}&\times&\FFx{q^r}\ar[rr]^{\theta}&&
    \FFx{q}&\times&\FFx{q^{pr}}\\
    x\ar@{|->}[dd]&&x_p\ar@/_25pt/@{|->}[dd]\ar@/^25pt/@{|->}[dd]&&
    x_r\ar@/_25pt/@{|->}[dd]\ar@/^25pt/@{|->}[dd]&&x_1&&x_{pr}\\
    &&T_1\times T_p&&T_1\times T_r&&&&
    \hspace{2cm}T_1\times T_p\times T_r\times T_{pr}\\
    x&&x_p^{\Phi_p(q)}\, ,\, x_p^{q-1} &&x_r^{\Phi_r(q)}\, ,\,
    x_r^{q-1}\ar@{|->}[rr]&&\,
    x_p^{\Phi_p(q)}
    \ar@{|->}[uu]&&\hspace{25pt}
    \qquad(x_r^{\Phi_r(q)},x_p^{q-1},x_r^{q-1}, x)
    \ar@{|->}@/^25pt/[uu]^{}\ar@{|->}@/_80pt/[uu]
  }
\end{displaymath}
\caption{Parametrization of $T_{pr}$}
\label{fig:mapping}
\end{figure}

We notice that the computation of this isogeny involves peculiar powers, which are based on evaluations in $q$ of modular inverses of cyclotomic polynomials. The values of their coefficients and the bounds of Theorem~\ref{th1} proven in Section~\ref{sec:inversion-phi_m-mod} ensure the low cost of this computation. We make use of a certain class of normal bases~\cite{CL09}, which allows efficient arithmetic in $\FF{q^n}$. We refer to~\cite{DL09} for more details.

For instance if we consider the example of $n=15=3\times 5$, then an explicit computation 
gives the following values, with the notations of Figure~\ref{fig:reconstruct}.

\begin{displaymath}
\begin{cases}
   u_1 = 1 \textrm{ and } u_{15} = -q^7-q^4-q^2-q,\\
   u_3 = -q \textrm{ and } u_5 = q^3+1,\\
   v_1 = 2q^8 - 2q^7 - 3q^6 + 8q^5 - 10q^4 + 6q^3 + 7q^2 - 
    16q + 9, \\
    v_2 = -2q^5 - 6q^4 - 9q^3 - 12q^2 - 10q - 6.
\end{cases}
\end{displaymath}





\begin{thebibliography}{99}
\bibliographystyle{jcm}
%
\bibitem{apo} T. M. Apostol. Resultants of Cyclotomic Polynomials. {\it
    Proceedings of the American Mathematical Society} (1970), Vol. 24, no. 3,
  pp. 457-462.
%
\bibitem{bach} G. Bachman. On the Coefficients of Ternary Cyclotomic
  Polynomials. {\it Journal of Number Theory} {\bf 100} (2003), pp. 104-116.
  %
\bibitem{bang} A. S. Bang. Om Ligningen $\phi_n(x)=0$. {\it Nyt Tidsskrift for
    Mathematic} {\bf 6} (1895), pp. 6-12.
  %
\bibitem{bateman} P.T. Bateman. Note on the coefficients of the cyclotomic polynomial. {\it Bulletin of the American Mathematical Society} {\bf
    55}, 1180-1181.
    %
\bibitem{beit} M. Beiter. Magnitude of the Coefficients of the Cyclotomic
  Polynomial $F_{pqr}$. {\it The American Mathematical Monthly} {\bf 75} (1968), pp. 370-372.
%
\bibitem{beit2} M. Beiter. Magnitude of the Coefficients of the Cyclotomic
  Polynomial $F_{pqr}$, II. {\it Duke math. J.} {\bf 38} (1971), pp. 591-594.
%
\bibitem{bloom} D. M. Bloom. On the Coefficients of The Cyclotomic
  Polynomials. {\it The American Mathematical Monthly} {\bf 75} (1968),
  pp. 372-377.
%
\bibitem{car} L. Carlitz. The Number of Terms in the Cyclotomic Polynomial
  $F_{pqr}$. {\it The American Mathematical Monthly} {\bf 73} (1966),
  pp. 979-981
  %
\bibitem{CL09} J.-M. Couveignes and R. Lercier. Elliptic Periods for Finite Fields, {\it Finite Fields and their Applications} {\bf 15}(1); pp. 1-22  
  %
\bibitem{COU} J.-M. Couveignes. Quelques math\'ematiques de la cryptologie \`a cl\'es publiques (Journ\'ee annuelle de la SMF), {\it Nouvelles m\'ethodes math\'ematiques pour la cryptographie}, Soci\'et\'e math\'ematique de France (2007).
%
\bibitem{DW04} M. van Dijk and D. Woodruff. Asymptotically Optimal
  Communication for Torus-Based Cryptography. {\it Crypto'04}, LNCS {\bf
    3152}, pp. 157-178.
%
\bibitem{DL09} C. Dunand and R. Lercier. Elliptic Bases and Torus-Based Cryptography. To be published in {\it Finite Fields and applications: proceedings of Fq9} (2009).
%
\bibitem{erdos} P. Erd\"os. On the Coefficients of the Cyclotomic
  Polynomial. {\it Bulletin of the American Mathematical Society} {\bf 52}
  (1946), pp. 179-184.
  %
  \bibitem{moree} Y. Gallot and P. Moree. Ternary Cyclotomic Polynomials Having
  a Large Coefficient. {\it Journal f\uu r die reine und angewandte Mathematik} {\bf 632} (2009), pp. 105-125.
  %
\bibitem{lam} T. Y. Lam and K. H. Leung. On the Cyclotomic Polynomial
  $\Phi_{pq}(X)$, {\it The American Mathematical Monthly} {\bf 103} (1996),
  pp. 562-564.
  %
  \bibitem{LV00}
A.~K. Lenstra and E.~R. Verheul, The XTR public key system, {\it Advances
  in Cryptology}, {CRYPTO} ' 2000 (Mihir Bellare, ed.), LNCS, {\bf 1880} (2000), pp.~1--19.
%
\bibitem{migo} A. Migotti. Zur Theorie der Kreisteilungsgleichung. {\it
    S.-B. der Math.-Naturwiss. Classe der Kaiserlichen Akademie der
    Wissenschaften, Wien} {\bf 87} (1883), pp 7-14.
  %
    \bibitem{mor} P. Moree. Inverse cyclotomic polynomials. {\it Journal of Number Theory} {\bf 129} (2009), pp. 667-680.
    %
\bibitem{RS03} K. Rubin and A. Silverberg. Torus-Based Cryptography. {\it
    Crypto'03}, LNCS {\bf 2729}, pp. 349-365.
    %
\bibitem{RS2} K. Rubin and A. Silverberg. Algebraic Tori in Cryptography. {\it High Primes Misdemeanours: Lectures in Honour of the 60th Birthday of Hugh Cowie Williams}, Fields Inst. Commun. {\bf 41}, AMS (2004), pp. 317-326.
%
\bibitem{SL93}
P.~J. Smith and M.~J. Lennon, {LUC}: {A} New Public Key System,
  {\it Proceedings of the {IFIP} /Sec '93}, Elsevier Science Publications, 1994.
  %
\bibitem{VOS}
V.~E. Voskresenski{\u\i}, {\it Algebraic {G}roups and {T}heir {B}irational
  {I}nvariants}, Translations of Mathematical Monographs, {\bf 179}, American
  Mathematical Societry, 1991.
%
\bibitem{WEI}
A. Weil, {\it Adeles and algebraic groups}, Progress in Math., {\bf 23}, Birkh\da user, Boston, 1982.
%
\bibitem{mca} J. von zur Gathen and J. Gerhard. {\it Modern Computer
    Algebra.} Cambridge University Press (1999).
%
\end{thebibliography}
\end{document}